\documentclass[11pt]{article}

\usepackage{amsfonts}
\usepackage{amssymb}
\usepackage{graphicx}
\usepackage{newlfont}
\usepackage{mathrsfs}
\usepackage{graphics}
\usepackage{amsmath, amssymb, amsfonts,amsthm}
\usepackage{textcomp}

\newtheorem{theorem}{Theorem}
\newtheorem{cor}[theorem]{Corollary}

\newtheorem{prop}[theorem]{Proposition}
\theoremstyle{remark}
\newtheorem{rem}{Remark}
\newtheorem{Example}{Example}

\newcommand{\bmu}{\boldsymbol \mu}
\newcommand{\blambda}{\boldsymbol \lambda}
\newcommand{\bzero}{\mathbf 0}

\newcommand{\cgp}{\stackrel{p}{\longrightarrow}}
\newcommand{\cgd}{\stackrel{d}{\longrightarrow}}

\newcommand{\ed}{\ \stackrel{d}{=} \ }

\newcommand{\Rbold}{\mbox{${\mathbb R}$}}
\newcommand{\Zbold}{\mbox{${\mathbb Z}$}}

\newcommand{\bE}{{\mathbf E}}
\newcommand{\bP}{{\mathbf P}}

\newcommand{\bx}{{\mathbf x}}
\newcommand{\by}{{\mathbf y}}
\newcommand{\bv}{{\mathbf v}}
\newcommand{\bu}{{\mathbf u}}

\newcommand{\eps}{\epsilon}

\title{{\bf Rate of Convergence and Large Deviation for the Infinite Color P\'olya Urn Schemes}}
\author{{\bf Antar Bandyopadhyay}\footnote{E-Mail: {\tt antar@isid.ac.in}}
        \footnote{Also affiliated with: Theoretical Statistics and Mathematics Unit, 
                                        Indian Statistical Institute, Kolkata;
                                        203 B. T. Road, Kolkata 700108, INDIA} \\ 
        {\bf Debleena Thacker}\footnote{E-Mail: {\tt thackerdebleena@gmail.com}} \vspace{0.25in} \\ 
        Theoretical Statistics and Mathematics Unit \\
         Indian Statistical Institute, Delhi Centre \\ 
         7 S. J. S. Sansanwal Marg \\
         New Delhi 110016 \\
         INDIA \vspace{0.15in} \\
				         }
\begin{document}
\maketitle
\begin{abstract}
In this work we consider the \emph{infinite color urn model} associated with a bounded increment random walk
on $\Zbold^d$. This model was first introduced in \cite{BaTh2013}. We prove that the rate of convergence of the
expected configuration of the urn at time $n$ with appropriate centering and scaling is of the order
${\mathcal O}\left(\frac{1}{\sqrt{\log n}}\right)$. Moreover we derive bounds similar to the classical Berry-Essen 
bound. Further we show that for the expected configuration a \emph{large deviation principle (LDP)} holds with 
a good rate function and speed $\log n$. 

\vspace{0.15in}

\vspace{0.15in}

\end{abstract}

\section{Introduction}
\label{Sec:Intro}
P\'olya urn scheme is one of the most well studied stochastic process which 
has plenty of applications in various different fields. Since the time of its introduction by
P\'olya \cite{Polya30} there has been a vast number of different variants and generalizations 
\cite{Fri49, Free65, BagPal85, Pe90, Gouet, Svante1, FlDuPu06, Pe07} studied in literature. 
In general one considers the model with finitely many colors and then it can be described simply by
\begin{quote}
Start with an urn containing finitely many balls of different colors. 
At any time $n\geq 1$, a ball is selected uniformly at random from the urn, 
and its color is noted. The selected ball is then returned to the urn along with a set of balls
of various colors which may depend on the color of the selected ball.
\end{quote}
In \cite{BlackMac73} Blackwell and MacQueen introduced a version of the model with possibly infinitely
many colors but with a very simple replacement mechanism. Recently the authors of this work
has introduced \cite{BaTh2013} a new generalization of the classical model with infinite but
countably many colors with replacement mechanism corresponding to random walks in $d$-dimension.
This generalization is essentially different than that of the classical P\'olya urn scheme, as well as
the model introduced in \cite{BlackMac73}, where the replacement mechanism is diagonal. 
The generalization by \cite{BaTh2013} considers replacement mechanism with 
non-zero off diagonal entries and provides
a novel connection between the two classical models, namely, P\'olya urn scheme
and random walks on $d$-dimensional Euclidean space has been demonstrated. In the current work
we exploit this connection to derive the \emph{rate of convergence} and 
the \emph{large deviation principle} for the
$\left(n+1\right)^{\mbox{th}}$ selected color in the infinite color generalization of the P\'olya
urn scheme. In the following subsection we 
describe the specific model which we study.

\subsection{Infinite Color Urn Model Associated with Random Walks}
\label{SubSec:Model}
Let $\left(X_j\right)_{j \geq 1}$ be i.i.d. random vectors taking values in $\Zbold^d$ with probability
mass function $p\left(\bu\right) := \bP\left(X_1 = \bu\right), \bu \in \Zbold^d$. We assume that the 
distribution
of $X_1$ is bounded, that is there exists a non-empty finite subset $B \subseteq \Zbold^d$ such that
$p\left(u\right) = 0$ for all $u \not\in B$. Throughout this paper 
we take the convention of writing all vectors as row vectors. 
Thus for a vector $\bx \in \Rbold^d$ we will write $\bx^T$ to denote it as a column vector.   
The notations  
$\langle \cdot , \cdot \rangle$ will denote 
the usual Euclidean inner product on $\Rbold^d$ and $\| \cdot \|$ the
the Euclidean norm. We will always write
\begin{equation}
\begin{array}{rcl}
\bmu & := & \bE\left[X_1\right] \\
\varSigma & := & \bE\left[ X_1^T X_1 \right] \\
e\left(\blambda\right) & := & \bE\left[e^{\langle \blambda , X_1 \rangle}\right], \, \blambda \in \Zbold^d. \\
\end{array}
\label{Equ:Basic-Notations}
\end{equation}
When the dimension $d=1$ we will denote the mean and variance simply by $\mu$ and $\sigma^2$ respectively.

Let $S_n := X_0 + X_1 + \cdots + X_n, n \geq 0$ be the random walk 
on $\Zbold^d$ starting at $X_0$ and with increments $\left(X_j\right)_{j \geq 1}$ which are independent. 
Needless to say that $\left(S_n\right)_{n \geq 0}$ is Markov chain with state-space $\Zbold^d$, 
initial distribution given by the distribution of $X_0$ and the transition matrix 
$R := \left(\left( p\left(\bu - \bv\right) \right)\right)_{u, v \in {\mathbb Z}^d}$. 

In \cite{BaTh2013} the following infinite color generalization of P\'olya urn scheme was introduced where
the colors were indexed by $\Zbold^d$.
Let $U_n := \left(U_{n,\bv}\right)_{\bv \in {\mathbb Z}^d} \in [0, \infty)^{{\mathbb Z}^d}$
denote the configuration of the urn at time $n$, that is, 
\[
\small{
\bP\left( \left(n+1\right)^{\mbox{th}} \mbox{\ selected ball has color\ } \bv 
\,\Big\vert\, U_n, U_{n-1}, \cdots, U_0 \right) 
\propto U_{n,\bv}, \, \bv \in \Zbold^d.
}
\]
Starting with $U_0$ which is a probability distribution we define $\left(U_n\right)_{n \geq 0}$
recursively as follows
\begin{equation}
\label{recurssion}
U_{n+1}=U_{n} + C_{n+1} R 
\end{equation}
where $C_{n+1} = \left(C_{n+1,\bv}\right)_{\bv \in {\mathbb Z}^d}$ is such that 
$C_{n+1,V}=1$ and $C_{n+1,\bu} = 0$ if $\bu \neq V$ where $V$ is a random color
chosen from the configuration $U_n$. In other words
\[
U_{n+1}=U_n + R_V
\]
where $R_V$ is the  $V^{\text{th}}$ row of the replacement matrix $R$. 
Following \cite{BaTh2013} we define 
the process $\left(U_n\right)_{n \geq 0}$ as the \emph{infinite color urn model} with
initial configuration $U_0$ and replacement matrix $R$. We will also refer it as the 
\emph{infinite color urn model associated with the random walk $\left(S_n\right)_{n \geq 0}$ on $\Zbold^d$}.
Throughout this paper we will assume that 
$U_0 = \left(U_{0,\bv}\right)_{\bv \in {\mathbb Z}^d}$ is such that 
$U_{0,\bv} = 0$ for all but finitely many $\bv \in \Zbold^d$.

It is worth noting that $\displaystyle{\sum_{\bu \in {\mathbb Z}^d} U_{n,\bu} = n + 1}$ for
all $n \geq 0$. So if
$Z_n$ denotes the $\left(n+1\right)^{\mbox{th}}$ selected color then
\begin{equation}
\bP\left(Z_n = \bv \,\Big\vert\, U_n, U_{n-1}, \cdots, U_0 \right) = \frac{U_{n,\bv}}{n+1}
\Rightarrow
\bP\left(Z_n = \bv \right) = \frac{\bE\left[U_{n,\bv}\right]}{n+1}.
\end{equation}
In other words the expected configuration of the urn at time $n$ is given by the distribution of $Z_n$.

\subsection{Outline of the Main Contribution of the Paper}
\label{SubSec:Outline}
In \cite{BaTh2013} the authors studied the asymptotic distribution of $Z_n$, in particular, it has been
proved (see Theorem 2.1 of \cite{BaTh2013}) that as $n \rightarrow \infty$,
\begin{equation}
\label{CLEC}
\frac{Z_{n}-\bmu \log n}{\sqrt{\log n}}\stackrel{d}{\longrightarrow}N_{d}\left(\bzero, \varSigma \right).
\end{equation}
In Section \ref{Sec:BE} we find the rate of convergence for the above asymptotic and show that 
classical Berry-Essen type bound hold at any dimension $d \geq 1$, which is of the order 
${\mathcal O}\left(\frac{1}{\sqrt{\log n}}\right)$.

It is easy to see that \eqref{CLEC} implies
\begin{equation}
\frac{Z_{n}}{\log n} \cgd \bmu \text{\ as\ } n \to \infty 
\Rightarrow
\frac{Z_{n}}{\log n} \cgp \bmu \text{\ as\ } n \to \infty.
\label{Equ:Zn-Prob-Conv}
\end{equation}
So it is then natural to ask whether the sequence of measures
$\left( \bP\left(\frac{Z_n}{\log n} \in \cdot \right)\right)_{n \geq 2}$
satisfy a \emph{large deviation principle (LDP)}. In Section \ref{Sec:LDP} we show that the above sequence of measures
satisfy a LDP with a good rate function and speed $\log n$. We also give an explicit representation
of the rate function in terms of rate function of a marked Poisson process with intensity one 
and the markings given by the i.i.d. increments $\left(X_j\right)_{j \geq 1}$. 

\subsection{Fundamental Representation}
\label{SunSec:Representation}
We end the introduction with the following very important observation made in \cite{BaTh2013} 
(see Theorem 3.1 in \cite{BaTh2013})
\begin{equation}
Z_n \ed Z_0 + \sum_{j=1}^n I_j X_j\,
\label{Equ:Representation}
\end{equation}
where 
$\left(X_j\right)_{j \geq 1}$ are as above and 
$\left(I_j\right)_{j \geq 1}$ are independent Bernoulli variables such that 
$I_j \sim \mbox{Bernoulli}\left(\frac{1}{j+1}\right)$ and are independent of $\left(X_j\right)_{j \geq 1}$.
$Z_0 \sim U_0$ and is independent of $\left(\left(X_j\right)_{j \geq 1}, \left(I_j\right)_{j \geq 1}\right)$.

Note that using this representation
the asymptotic normality \eqref{CLEC} follows immediately as an application of the 
Lindeberg Central Limit Theorem \cite{Bill86}. 
We use this representation to derive the Berry-Essen type bounds and also the LDP.

\section{Berry-Essen Bounds for the Expected Configuration}
\label{Sec:BE}
In this section we show that the rate of convergence of (\ref{CLEC}) is of the order
${\mathcal O}\left( \frac{1}{\sqrt{\log n}} \right)$. In fact we show that the Berry-Essen
type bound holds for the color of the $\left(n+1\right)^{\mbox{th}}$-selected ball.

\subsection{Berry-Essen Bound for $d=1$}
\label{SubSec:BE-1}
We first consider the case when the associated random walk is a one dimensional walk and the set of colors
are indexed by the set of integers $\Zbold$.
\begin{theorem}
\label{Thm:BE-1}
Suppose $U_0 = \delta_0$ then
\begin{equation} 
\sup_{x \in {\mathbb R}} 
\left\vert \bP\left( \frac{Z_{n}-\mu h_n}{\sqrt{n \rho_2}} \leq x\right) - \Phi\left(x\right) \right\vert 
\leq 
2.75 \times \frac{\sqrt{n} \rho_3}{\rho_2^{3/2}} = {\mathcal O} \left( \frac{1}{\sqrt{\log n}} \right),
\label{Equ:BE-1}
\end{equation}
where $\displaystyle{h_n := \sum_{j=1}^n \frac{1}{j+1}}$, $\Phi$ is the standard normal distribution 
function and
\begin{equation}
\rho_2 := \frac{1}{n} \left( \sigma^2 h_n - 
                             \mu^2 \sum_{j=1}^n \frac{1}{\left(j+1\right)^2} \right) 
\label{Equ:Def-rho2}                             
\end{equation}
and
\begin{equation}
\rho_3 := \frac{1}{n} \left(
          \sum_{j=1}^n \frac{1}{j+1} \bE\left[ \left\vert X_1 - \frac{\mu}{j+1} \right\vert^3 \right]
          + \left\vert \mu \right\vert^3 \sum_{j=1}^n \frac{j}{\left(j+1\right)^4} \right).
\label{Equ:Def-rho3}
\end{equation}
\end{theorem}

\begin{proof}
We first note that when $U_0 = \delta_0$ then \eqref{Equ:Representation} can be written as
\begin{equation}
Z_n \ed \sum_{j=1}^n I_j X_j\,
\label{Equ:Representation-0}
\end{equation}
where $\left(X_j\right)_{j \geq 1}$ are i.i.d. increments of the random walk $\left(S_n\right)_{n \geq 0}$,
$\left(I_j\right)_{j \geq 1}$ are independent Bernoulli variables such that 
$I_j \sim \mbox{Bernoulli}\left(\frac{1}{j+1}\right)$ and are independent of $\left(X_j\right)_{j \geq 1}$.

Now observe that 
\[
n \rho_2 = \sum_{j=1}^n \bE\left[ \left(I_j X_j - \bE\left[I_j X_j\right] \right)^2\right]
\mbox{\ and\ } 
n \rho_3 = \sum_{j=1}^n \bE\left[ \left\vert I_j X_j - \bE\left[I_j X_j\right] \right\vert^3\right].
\]
Thus from the \emph{Berry-Essen Theorem} for the independent but non-identical increments
(see Theorem 12.4 of \cite{BhRa76}) we get
\begin{equation}
\sup_{x \in {\mathbb R}} 
\left\vert \bP\left( \frac{\sum_{j=1}^n I_j X_j -\mu h_n}{\sqrt{n \rho_2}} \leq x\right) 
- \Phi\left(x\right) \right\vert 
\leq 
2.75 \times \frac{\sqrt{n} \rho_3}{\rho_2^{3/2}}.
\label{Equ:BE-Intermediate}
\end{equation}
The equations \eqref{Equ:Representation-0} and \eqref{Equ:BE-Intermediate} implies the inequality in 
\eqref{Equ:BE-1}.

Finally to prove the last part of the equation \eqref{Equ:BE-1} we note that from definition
$n \rho_2 \sim C_1 \log n$ and $n \rho_3 \sim C_2 \log n$ where $0 < C_1, C_2 < \infty$ are some constants.
Thus
\[
\frac{\sqrt{n} \rho_3}{\rho_2^{3/2}} = {\mathcal O}\left(\frac{1}{\sqrt{\log n}}\right).
\]
This completes the proof of the theorem. 
\end{proof}

Following result follows easily from the above theorem by observing the facts $h_n \sim \log n$ and
$n \rho_2 \sim C_1 \log n$. 
\begin{theorem}
\label{Thm:BE-1-General}
Suppose $U_{0,k} = 0$ for all but finitely many $k \in \Zbold$ then
there exists a constant $C > 0$ such that
\begin{equation} 
\sup_{x \in {\mathbb R}} 
\left\vert \bP\left( \frac{Z_{n}-\mu \log n}{\sigma \sqrt{\log n}} \leq x \right) 
- \Phi\left(x\right) \right\vert 
\leq 
C \times \frac{\sqrt{n} \rho_3}{\rho_2^{3/2}} = {\mathcal O} \left( \frac{1}{\sqrt{\log n}} \right),
\label{Equ:BE-1-General}
\end{equation}
$\Phi$ is the standard normal distribution function and
$\rho_2$ and $\rho_3$ are as defined in \eqref{Equ:Def-rho2} and \eqref{Equ:Def-rho3} respectively.
\end{theorem}
It is worth noting that unlike in Theorem \ref{Thm:BE-1}
the constant $C$ which appears in \eqref{Equ:BE-1-General} above, is not a universal constant, it may
depend on the increment distribution, as well as on $U_0$.

\subsection{Berry-Essen bound for $d \geq 2$}
\label{SubSec:BE-d}
We now consider the case when the associated random walk is $d \geq 2$ dimensional and the colors are
indexed by $\Zbold^d$. Before we present our main result we introduce few notations. 

{\bf Notations:} For a vector $\bx \in \Rbold^d$ we will write the coordinates as  
$\left(x^{(1)}, x^{(2)}, \cdots, x^{(d)}\right)$. For example the coordinates of
$\bmu$ will be written as $\left(\mu^{(1)}, \mu^{(2)}, \cdots, \mu^{(d)}\right)$.
For a matrix $A = \left(\left(a_{ij}\right)\right)_{1\leq i,j\leq d }$ we denote by 
$A\left(i,j\right)$ the $(d-1) \times (d-1)$ sub-matrix of $A$, obtained by deleting the $i^{\text{th}}$ 
row and $j^{\text{th}}$ column.  
Let
\begin{equation}
\rho_2^{(d)} := \frac{1}{n} \sum_{j=1}^{n} \frac{1}{(j+1)}\frac{\det\left(\varSigma -\frac{1}{j+1}M\right)}
                    {\det\left(\varSigma(1,1)-\frac{1}{j+1}M(1,1)\right)},
\label{Equ:Def-rho2-d}
\end{equation}
where $M := \left(\left(\mu^{(i)} \mu^{(j)}\right)\right)_{1 \leq i, j \leq d}$ and
\begin{equation}
\rho_3^{(d)} := \frac{1}{n d} \sum_{j=1}^n \sum_{i=1}^{d}\gamma^{3}_{n}\left(i\right) 
                                                                   \beta_{j}\left(i\right),
\label{Equ:Def-rho3-d}                                                                   
\end{equation}
where 
\[
\gamma^{2}_{n}(i) :=
\max_{1\leq j \leq n} 
\frac{\det\left(\varSigma(i,i)-\frac{1}{(j+1)}M(i,i)\right)}
     {\det\left(\varSigma(1,1)-\frac{1}{j+1}M(1,1)\right)}
\]
and
\[ 
\beta_{j}(i) = 
\frac{1}{j+1} \bE\left[ \left\vert X_1^{(i)} - \frac{\mu^{(i)}}{j+1} \right\vert^3 \right]
+ \frac{j}{\left(j+1\right)^4} \left\vert \mu^{(i)} \right\vert^3.
\]
For any two vectors $\bx$ and $\by \in \Rbold^d$ we will write $\bx \leq \by$, if the inequality holds 
coordinate wise. Finally for a positive definite matrix $B$,
we write $B^{1/2}$ for the unique positive definite square root of it.

\begin{theorem}
\label{Thm:BE-d}
Suppose $U_0 = \delta_{\bzero}$ then there exists an universal constant $C\left(d\right) > 0$ which 
may depend on the dimension $d$ such that
\begin{equation}
\sup_{\bx \in {\mathbb R}^d} 
\left\vert \bP\left( \left(Z_{n}-\bmu h_n\right) \varSigma_n^{-1/2} \leq \bx \right) 
- \Phi_d\left(\bx\right) \right\vert 
\leq 
C\left(d\right) \frac{\sqrt{n} \rho_3^{(d)}}{\left(\rho_2^{(d)}\right)^{3/2}} 
= {\mathcal O} \left( \frac{1}{\sqrt{\log n}} \right),
\label{Equ:BE-d}
\end{equation}
where $\varSigma_n := \sum_{j=1}^n \frac{1}{j+1} \left(\varSigma - \frac{1}{j+1} M\right)$ and
$\Phi_d$ is the distribution function of a standard $d$-dimensional normal random vector.
\end{theorem}

\begin{proof} 
Like in the one dimensional case, we start by observing that
when $U_0 = \delta_0$ then \eqref{Equ:Representation} can be written as
\begin{equation}
Z_n \ed \sum_{j=1}^n I_j X_j\,
\label{Equ:Representation-0-d}
\end{equation}
where $\left(X_j\right)_{j \geq 1}$ are i.i.d. increments of the random walk $\left(S_n\right)_{n \geq 0}$,
$\left(I_j\right)_{j \geq 1}$ are independent Bernoulli variables such that 
$I_j \sim \mbox{Bernoulli}\left(\frac{1}{j+1}\right)$ and are independent of $\left(X_j\right)_{j \geq 1}$.

Now the proof of the inequality in \eqref{Equ:BE-d} follows from equation (D) of \cite{Bergstrom49}
which deals with $d$-dimensional version of the classical Berry-Essen inequality for independent but
non-identical summands, which in our case are the random variables $\left(I_j X_j\right)_{j \geq 1}$. 
It is enough to notice that
\[
\beta_{j}(i) = \bE\left[ \left\vert I_j X_1^{(i)} - \bE\left[I_j X_j^{(i)}\right] \right\vert^3 \right],
\]
and
\[
\varSigma_n = \sum_{j=1}^n 
\bE\left[ 
\left( I_j X_j - \bE\left[I_j X_j\right]\right)^T \left(I_j X_j - \bE\left[I_j X_j\right] \right)
\right].
\]

Finally to prove the last part of the equation \eqref{Equ:BE-d} just like in the one dimensional case
we note that from definition
$n \rho_2^{(d)} \sim C_1' \log n$ and $n \rho_3^{(d)} \sim C_2' \log n$ where 
$0 < C_1', C_2' < \infty$ are some constants. Thus
\[
\frac{\sqrt{n} \rho_3}{\rho_2^{3/2}} = {\mathcal O}\left(\frac{1}{\sqrt{\log n}}\right).
\]
This completes the proof of the theorem. 
\end{proof}

\begin{rem} 
If we define that $\varSigma\left(1,1\right) = 1$ and $M\left(1,1\right)=0$ when $d=1$ then Theorem 
\ref{Thm:BE-1} follows from the above theorem except in Theorem \ref{Thm:BE-1} the constant is more explicit. 
\end{rem}

Just like in the one dimensional case the following result follows easily from the above theorem
by observing $h_n \sim \log n$. 
\begin{theorem}
\label{Thm:BE-d-General}
Suppose $U_0 = \left(U_{0,\bv}\right)_{\bv \in {\mathbb Z}^d}$ is such that 
$U_{0,\bv} = 0$ for all but finitely many $\bv \in \Zbold^d$ then there exists a constant 
$C > 0$ which may depend on the increment distribution, such that
\begin{equation} 
\sup_{\bx \in {\mathbb R}^d} 
\left\vert \bP\left( \left(\frac{Z_{n}-\bmu \log n}{\sqrt{\log n}} \right) \varSigma^{-1/2} \leq \bx \right) 
- \Phi_d\left(\bx\right) \right\vert 
\leq 
C \times \frac{\sqrt{n} \rho_3^{(d)}}{\left(\rho_2^{(d)}\right)^{3/2}} 
= {\mathcal O} \left( \frac{1}{\sqrt{\log n}} \right),
\label{Equ:BE-d-General}
\end{equation}
where $\Phi_d$ is the distribution function of a standard $d$-dimensional normal random vector.
\end{theorem}

\section{Large Deviations for the Expected Configuration}
\label{Sec:LDP}
In this section we discuss the asymptotic  behavior of the tail probabilities of $\frac{Z_{n}}{\log n}$.
Following standard notations are used in rest of the paper. For any subset $A \subseteq \Rbold^d$ we write 
$A^{\circ}$ to denote the \emph{interior} of $A$ and $\bar{A}$ to denote the \emph{closer} of $A$ under the
usual Euclidean metric. 
\begin{theorem}
\label{Thm:LDP} 
The sequence of
measures $\bP\left(\frac{Z_n}{\log n} \in \cdot \right)_{n \geq 2}$ satisfy a LDP
with rate function $I\left(\cdot\right)$ and speed $\log n$, that is, 
\begin{equation}
\small{
- \inf_{\bx \in  A^{\circ}} I\left(\bx\right) \leq 
\mathop{\underline{\lim}}\limits_{n \to \infty} \frac{\log \bP\left(\frac{Z_{n}}{\log n}\in A\right)}{\log n} \leq 
\mathop{\overline{\lim}}\limits_{n \to \infty} \frac{\log \bP\left(\frac{Z_{n}}{\log n}\in A\right)}{\log n} \leq 
- \inf_{\bx \in \bar{A}} I\left(\bx\right)
}
\label{Equ:LDP}
\end{equation} 
\normalsize
where $I(\cdot)$ is the Fenchel-Legendre dual of $e\left(\cdot\right)-1$, that is for $x \in \mathbb{R}^{d}$,
\begin{equation}
I(x)=\displaystyle\sup_{\blambda \in \mathbb{R}^{d}}\{\langle \bx, \blambda \rangle -e(\blambda) + 1 \}.
\label{Equ:Def-I}
\end{equation}
Moreover $I(\cdot)$ is convex and a \emph{good rate function}.
\end{theorem}

\begin{proof}
We start with the representation \eqref{Equ:Representation}
\[
Z_n \ed Z_0 + \sum_{j=1}^n I_j X_j
\]
where as earlier
$\left(X_j\right)_{j \geq 1}$ are i.i.d. increments of the random walk $\left(S_n\right)_{n \geq 0}$ on
$\Zbold^d$ and 
$\left(I_j\right)_{j \geq 1}$ are independent Bernoulli variables such that 
$I_j \sim \mbox{Bernoulli}\left(\frac{1}{j+1}\right)$ and are independent of $\left(X_j\right)_{j \geq 1}$.
$Z_0 \sim U_0$ and is independent of $\left(\left(X_j\right)_{j \geq 1}, \left(I_j\right)_{j \geq 1}\right)$.
Now without loss of any generality we may assume that $Z_0 = \bzero$ with probability one, that is, 
$U_0 = \delta_{\bzero}$.

Consider the following scaled \emph{logarithmic moment generating function} of $Z_n$,
\begin{equation}
\Lambda_{n}\left(\blambda\right) := \frac{1}{\log n}\log \mathbb{E}\left[e^{\langle \blambda, Z_{n}\rangle}\right].
\label{Equ:log-MGF}
\end{equation}
From \eqref{Equ:Representation} it follows that 
\[
\mathbb{E}\left[e^{\langle \blambda, Z_{n}\rangle}\right]=\frac{1}{n+1}\Pi_{n}\left(e\left(\blambda\right)\right)
\] 
where $\Pi_{n}\left(z\right)=\prod_{j=1}^{n}\left(1+\frac{z}{j}\right)$, $z \in \mathbb{C}$.
Using Gauss's formula (see page 178 of \cite{Con78}) we have  
\begin{equation}
\lim_{n\to \infty} \frac{\Pi_{n}(z)}{n^{z}}\Gamma(z+1)=1
\label{Equ:Gauss}
\end{equation} 
and the convergence happens uniformly on compact subsets of $\mathbb{C}\setminus\{-1,-2,\ldots\}$.
Therefore we get
\begin{equation}
\Lambda_{n}\left(\blambda\right) \longrightarrow e\left(\blambda\right) - 1 < \infty \,\,\, \forall \,\,
\blambda \in \Rbold^d.
\end{equation}
Thus the LDP as stated in \eqref{Equ:LDP} follows from the G\"artner-Ellis Theorem 
(see Remark (a) on page 45 of \cite{DeZe1993} or page 66 of \cite{ThesisArijit2010}).

We next note that 
$I(\cdot)$ is a convex function because it is the Fenchel-Legendre dual of 
$e\left(\blambda\right)-1$ which is finite for all $\blambda \in \mathbb{R}^{d}$.

Finally, we will show that $I\left(\cdot\right)$ is good rate function, that is, 
the level sets $A\left(\alpha \right)=\{\bx\colon I(\bx)\leq \alpha\}$ are compact for all 
$\alpha > 0$. 
Since $I$ is a rate function so by definition it is lower semicontinuous. So it is enough to prove that 
$A(\alpha)$ is bounded for all $\alpha \in \Rbold$. 

Observe that for all $\bx \in \mathbb{R}^{d}$,
\[
I(\bx) \geq \sup_{\| \blambda \| = 1}\left\{\langle \bx,\blambda \rangle-e(\blambda)+1\right\}.
\]
Now the function $\blambda \mapsto e\left(\blambda\right)$ 
is continuous and $\left\{\blambda \colon \| \blambda \| = 1 \right\}$ 
is a compact set. So $\exists$ $\blambda_{0} \in \left\{\blambda \colon \| \blambda \| = 1\right\}$ 
such that $\sup_{\vert\blambda\vert=1} e\left(\blambda\right) = e\left(\blambda_{0}\right)$. 
Therefore  for $\| \bx \| \neq 0$ choosing $\blambda=\frac{\bx}{\| \bx \|}$, we have 
$ I(\bx) \geq \| \bx \| -e\left(\blambda_{0}\right)+1$. So
if $\bx\in A(\alpha)$ then
\[
\| \bx \| \leq \left(\alpha +e \left(\blambda _{0}\right)-1\right).
\]
This proves that the level sets are bounded, which completes the proof.
\end{proof}

Our next result is an easy consequence of \eqref{Equ:Def-I} which can be used to compute explicit
formula for the rate function $I$ in many examples in one or higher dimensions.
\begin{theorem}
\label{Thm:I-Formula}
The rate function $I$ is same as the rate function for the large deviation of the
empirical means of i.i.d. random vectors with distribution 
corresponding to the distribution of the following random vector
\begin{equation}
W = \sum_{i=1}^N X_i,
\label{Equ:I-Representation}
\end{equation}
where $N \sim \mbox{Poisson}\left(1\right)$ and is independent of 
$\left(X_j\right)_{j \geq 1}$ which are the i.i.d. increments of the associated random walk. 
\end{theorem}

\begin{proof}
We first observe that 
$\log \bE\left[e^{\langle \blambda, W \rangle}\right] = e\left(\blambda\right) - 1$. The rest
then follows from \eqref{Equ:Def-I} and Cram\'er's Theorem (see Theorem 2.2.30 of \cite{DeZe1993}).  
\end{proof}

\begin{rem}
Using Theorem \ref{Thm:I-Formula} we can conclude that   
the tail of the asymptotic distribution of $Z_n$ can be approximated by the tail of a 
marked Poisson process with intensity one where the markings are given by the i.i.d. 
increments of the associated random walk. 
\end{rem}

For $d=1$, one can get more information about the rate function $I$, in particular 
following result it follows from Theorem \ref{Thm:I-Formula} and Lemma 2.2.5 of \cite{DeZe1993}.
\begin{prop}
\label{Cor:LDP1}
Suppose $d=1$ then $I(x)$ is non-decreasing when $x \geq \mu$ and non-increasing when $x \leq \mu$. Moreover 
\begin{equation}
I(x) = \begin{cases}
       \displaystyle{\sup _{\lambda \geq 0}\{x\lambda -e(\lambda)+1\}} & \mbox{if\ \ } x \geq \mu \\
       \displaystyle{\sup _{\lambda \leq 0}\{x\lambda -e(\lambda)+1\}} & \mbox{if\ \ } x \leq \mu. 
       \end{cases}
\label{incr}
\end{equation}
In particular, $I(\mu)=\inf_{x \in \mathbb{R}} I(x)$.
\end{prop}
% 
% \begin{proof}
% We start by observing that the function  
% $e\left(\lambda\right)=\textstyle \sum_{j \in B}e^{\lambda j}p(j)$ is differentiable in
% $\lambda$ as we have assumed that $B$ is a finite subset of $\mathbb{Z}$. 
% Let $f_{x}\left(\lambda\right)=x \lambda -e\left(\lambda\right)+1$. Differentiating $f_{\mu}\left(\lambda\right)$ with respect to $\lambda$ and equating to $0$ we obtain,
% \begin{eqnarray*}
% \mu=\displaystyle \sum_{v \in B} ve^{\lambda v}p(v).
% \end{eqnarray*}. Moreover since $e\left(\lambda\right)$ is convex, $\frac{d^{2}}{d \, \lambda ^{2}}f_{\mu}\left(\lambda\right)<0$Therefore $I(\mu)=0$. For $x\geq \mu $ and $\lambda<0$ we have $x\lambda -e\left(\lambda\right)+1
% \leq \mu\lambda -e\left(\lambda\right)+1\leq I(\mu)=0$. Therefore $I(x)$ is given by (\ref{incr}). For $x,y \geq \mu$, with $x\geq y $ we have for all $\lambda >0$, $x\lambda -e\left(\lambda\right)+1
% \geq y\lambda -e\left(\lambda\right)+1=0$. 
% This together with (\ref{incr}) proves that $I(\cdot)$ is non 
% decreasing for $x \geq \mu$. The proof for $x\leq \mu$ follows similarly. 
% \end{proof}

Following is an immediate corollary of the above result and Theorem \ref{Thm:LDP}.
\begin{cor}\label{Th: CramerLDP}
Let $d=1$ then for any $\eps > 0$
\begin{equation}
\label{Eq: lim}
\lim_{n \to \infty} \frac{1}{\log n} \log \bP \left(\frac{Z_{n}}{\log n}\geq \mu + \eps \right) 
= - I\left( \mu + \eps \right)
\end{equation}
and
\begin{equation}
\lim_{n \to \infty} \frac{1}{\log n} \log \bP \left(\frac{Z_{n}}{\log n} \leq \mu - \eps \right)
= -I\left( \mu - \eps \right).
\end{equation}
\end{cor}

We end the section with explicit computations of the rate functions for two examples of 
infinite color urn models
associated with random walks on one dimensional integer lattice. 
\begin{Example} 
Our first example is the case when the random walk is trivial, which moves deterministically one step at 
a time. In other words $X_1 = 1$ with probability one. In this case $\mu = 1$ and $\sigma^2=1$. Also
the moment generating function of $X_1$ is given by 
$e\left(\lambda\right) := e^{\lambda}$, $\lambda \in \Rbold$. 
By Theorem \ref{Thm:I-Formula} 
the rate 
function for the associated infinite color urn model is same as the rate function for a Poisson random
variable with mean $1$, that is
\begin{equation}
I(x)=
\begin{cases}
+ \infty & \text{\ if\ } x < 0 \\
1 & \text{\ if\ } x = 0 \\
x \log x - x + 1 & \text{\ if\ } x > 0
\end{cases}
\end{equation}
Thus for this example one can prove a \emph{Poisson approximation} for $Z_n$. 
\end{Example}

\begin{Example} 
Our next example is the case when the random walk is the \emph{simple symmetric random walk} on the one 
dimensional integer lattice. For this case we note that $\mu = 0$, $\sigma^2 = 1$ and the 
moment generating function $X_1$ is 
$e\left(\lambda\right) = \cosh \lambda$, $\lambda \in \Rbold$. The rate 
function for the associated infinite color urn model turns out to be 
\begin{equation}
I(x)=x \sinh^{-1} x - \sqrt{1+x^2} + 1.
\end{equation}
\end{Example}

\bibliographystyle{plain}

\bibliography{Berry-Essen}

\end{document}